\documentclass[final]{IEEEtran}

\IEEEoverridecommandlockouts              

\usepackage{graphicx}

\usepackage{url}
\usepackage{amsmath,amssymb}
\usepackage{amsthm,thmtools}
\usepackage{amsfonts} 

\usepackage{enumerate}
\usepackage{color}                  

\usepackage{verbatim}

\usepackage{amssymb}
\usepackage{amsbsy}
\usepackage{amsmath}
\usepackage{setspace}
\usepackage{url}

\usepackage{float}
\usepackage{dsfont}
\usepackage{siunitx}    
\usepackage{tikz}

\usetikzlibrary{
	arrows.meta,
	positioning,
	shapes.geometric,
	calc
}

\usepackage{epsfig}

 \newcommand{\diag}{\operatorname{diag}}

  \newcommand{\rank}{\operatorname{rank}}
 
 \newcommand{\inv}{\operatorname{inv}}
\newcommand{\st}{\, | \,}

 \DeclareMathOperator*{\argmin}{arg\,min}

\newcommand{\spanop}{\operatorname{span}}

 \newtheorem{Theorem}{Theorem}
 \newtheorem{Lemma}{Lemma}

 \newtheorem{Corollary}{Corollary}
  \newtheorem{Definition}{Definition}
 \newtheorem{Problem}{Problem}
 \newtheorem{Proposition}[Theorem]{Proposition}
 \newtheorem{Remark} {Remark}
 \newtheorem{Example} {Example}

\newcommand {\R}{\mathbb R}

\newcommand{\be}{\begin{equation}}
\newcommand{\ee}{\end{equation}}
\newcommand{\ben}{\begin{equation*}}
\newcommand{\een}{\end{equation*}}

\newcommand{\updt}[1]{{\color{black}#1}}

\begin{document}

\title{On IDA-PBC with Maximum Energy Shapeability
\thanks{
The work of C. Wu and W. Liu was supported by the National Natural Science Foundation of China under Grant 62403370. (\emph{$^*$Corresponding author: Chengshuai Wu.})
}}
\author{Chengshuai Wu$^*$, Wenbin Liu,  Meng Zhang, Jianchen Hu, and Romeo Ortega 
\thanks{
\IEEEcompsocthanksitem
	C. Wu, W. Liu, and J. Hu are with the School of Automation Science and Engineering, Xi'an Jiaotong University, Xi'an 710049, China (e-mails: chengshuai.wu@xjtu.edu.cn, horace89@xjtu.edu.cn). 
\IEEEcompsocthanksitem
	M. Zhang is with the School of Cyber Science and Engineering, Xi'an Jiaotong University, Xi'an 710049, China (e-mail: mengzhang2009@xjtu.edu.cn).
\IEEEcompsocthanksitem
   R. Ortega is with the Department of Electrical Engineering, Instituto Tecnol\'ogico Aut\'onomo de M\'exico (ITAM), Mexico city 01080, Mexico.	(e-mail: romeo.ortega@itam.mx). 
	}}
	
	\maketitle

\begin{abstract}
Interconnection and Damping Assignment Passivity-Based Control (IDA-PBC) is a well-established stabilization technique for affine nonlinear systems. However, its application is generally hindered by the requirement of solving a set of partial differential equations (PDEs), i.e., the so-called matching equation. 
This paper \updt{explores} the notion of \emph{maximum energy shapeability} which describes the scenario that the homogeneous part of the matching equation admits $m$ independent solutions with $m$ the dimension of the control input. We show that the maximum energy shapeability enables a systematic procedure for the IDA-PBC design by transforming the matching equation into a set of new PDEs. \updt{Specifically, the transformed matching equation admits an explicit solution in the single input case.}
Sufficient conditions for maximum energy shapeability are also provided. It is shown that some existing constructive IDA-PBC designs actually implicitly exploit the maximum energy shapeability. The proposed procedure for the IDA-PBC design is  \updt{demonstrated with two practical examples}.
\end{abstract}


\section{INTRODUCTION}
Interconnection and Damping Assignment Passivity-Based Control (IDA-PBC) is first introduced in \cite{ortega2001putting} to stabilize a class of physical systems described in the port-Hamiltonian (pH) form~\cite{van2014port}. It is later extended to general affine nonlinear systems and especially underactuated ones  \cite{ortega2002interconnection}. Unlike some classical nonlinear control methods such as feedback linearization, or the backstepping procedure for strict-feedback systems (see e.g., \cite{sepulchre2012constructive}, \cite{hkhalil2002}, \cite{isidori-book}), the IDA-PBC method directly exploits the inherent dynamics of the controlled systems, leading to a control design without relying on nonlinearity cancellation or high gain design. 
The core idea of the IDA-PBC method is the so-called \textit{energy shaping} such that the closed-loop system achieves a pH structure with certain desired Hamiltonian/energy function, interconnection matrix and damping matrix, which are the solutions to a set of partial differential equations (PDEs), namely, the \textit{matching equation}. 
A major bottleneck in applying IDA-PBC is solving the matching PDEs, for which no systematic approach exists in general.

Some researchers dedicated to provide closed-form solutions of the matching equations for systems with special structures such as mechanical systems. It is shown in \cite{coordinate2007} that the matching equation of mechanical systems, which consists of two sets of PDEs corresponding to the kinetic energy and the potential energy, can be simplified by parameterizing the target dynamics and conducting appropriate coordinate changes. In \cite{acosta2005interconnection}, the authors investigated the mechanical systems with underactuation degree of one. It shows that the kinetic energy PDEs, which are nonlinear and inhomogeneous, can be transformed into algebraic ones with a predefined desired inertia matrix, while the potential energy PDEs can be explicitly solved provided that the original inertia matrix and the potential energy function are independent of the underactuated coordinates. 
The results in~\cite{acosta2005interconnection} are further extended in~\cite{ryalat2016simplified} where it is assumed that the inertia matrix and the potential energy function depend only on one underactuated/actuated coordinate.
\updt{For underactuated $2$-degree-of-freedom (2-DoF) mechanical systems, explicit constructions of the desired inertia matrix, the interconnection matrix, and  the potential energy are derived in~\cite{arpenti2022constructive}.}

Another line of research on IDA-PBC focuses on circumventing the requirement of solving the matching equation. As firstly proposed in~\cite{fujimoto2001canonical}, it is well known that the matching equation reduces to a set of algebraic equations when the desired energy function is pre-determined.
In \cite{nunna2015constructive}, the IDA-PBC design for pH-systems is modified by introducing extended state variables, and this yields a dynamic feedback law. It obviates the need for solving the PDEs which are replaced by a set of nonlinear algebraic equations. For a class of pH-systems whose interconnection, damping, and input matrices satisfy certain integrability condition for vector fields, i.e., the \emph{Poincar\'e's Lemma}, Ref. \cite{borja2016constructive} shows that an IDA-PBC controller can be designed without directly relying on the solution of the matching equation. The similar integrability condition is also adopted in~\cite{Donaire2016shaping} which provides an energy shaping method without explicitly solving PDEs for a class of mechanical systems described by the Euler-Lagrange equations.

In this paper we seek to transform the matching equation, by exploiting its properties, into a form potentially easier to solve. Similar ideas have been employed in existing works, for example, the matching equation can be further simplified using the \emph{Poincar\'e's Lemma} by factorization of the desired energy function~\cite{cheng2005feedback} or the control law~\cite{wu2020stabilization}. In~\cite{kotyczka2013local}, it is shown that a reduced order matching equation can be obtained by introducing a change of coordinates, which exists provided the homogeneous part of the matching equation admits $m$ independent solutions, with $m$ being the dimension of the control input. Based on~\cite{kotyczka2013local}, we formally defines the notion of \emph{maximum energy shapeability} to specify the case where such $m$ solutions exist and the matching equation is solvable. In this case, we show that the resulting reduced order matching equation can be further transformed by exploiting the higher-order \emph{mixed-partial symmetry requirement} for the Hessian of the desired energy function.
Then a systematic procedure for the IDA-PBC design is proposed based on the \emph{Mean Value Theorem for vector-valued functions}. \updt{For single-input control-affine nonlinear systems, the proposed procedure is constructive, since the transformed PDEs admit explicit solutions.}

The remainder of this paper is organized as follows. The next section reviews some basic results on the IDA-PBC method. The main results are given in Section~\ref{sec:main}. Section~\ref{sec:suff} presents a set of sufficient conditions to ensure maximum energy shapeability. In Section~\ref{sec:simu}, the proposed procedure for the IDA-PBC design is illustrated with two application examples. Section~\ref{sec:conclu} concludes the paper.

\section{Preliminaries} \label{sec:pre}
 Consider the affine nonlinear system
\be  \label{eq:sys}
\dot x =  f(x) + G(x) u
\ee
where $x \in \mathbb{R}^n$ is the state, $u \in \mathbb{R}^m$, with $m < n$, is the control input, that is, the system~\eqref{eq:sys} is underactuated. The mappings $f : \mathbb{R}^n \to \mathbb{R}^n$, $G: \mathbb{R}^n \to \mathbb{R}^{n \times m} $ are continuous and sufficiently differentiable. It is assumed that $\rank(G(x)) = m$ for all $x \in \Omega$, where $\Omega \subset \R^n$ denotes a simply connected state domain under concern. 
Note that all the assignable equilibrium points  are contained in the set 
$ 
\mathcal{E} := \{  x \in \Omega ~|~ G^\perp(x) f(x) = 0 \}
$
where $G^\perp: \R^n \to \R^{(n-m) \times n}$ is a full-rank left annihilator of $G(x)$, i.e., $G^\perp (x) G(x) = 0$. 

The standard IDA-PBC approach is to design a state feedback such that the closed-loop system of~\eqref{eq:sys} takes the port-Hamiltonian (pH) form \cite{ortega2002interconnection}
\be \label{eq:pch}
\dot x = (J_d(x) - R_d(x)) \nabla H_d(x)
\ee
where the matrices $J_d, R_d: \R^n \to \R^{n \times n}$ satisfy $J_d^T(x) = -J_d(x)$, $R_d(x) =  R_d^T(x) \succeq 0$, and are referred to as the interconnection and damping matrices, respectively. For simplicity, we define $F_d(x) :=  J_d(x) -R_d(x)$ throughout the paper. The mapping $H_d: \R^n \to \R$ is the energy function with a (local) minima at a desired equilibrium point $x^* \in \mathcal{E}$, i.e.,
\be  \label{eq:mini}
x^* = \argmin H_d(x).
\ee
In summary, the IDA-PBC design is to solve the next problem.

\begin{Problem} \label{prb:ida}
For a desired equilibrium point $x^* \in \mathcal{E}$ of the system~\eqref{eq:sys},
pick a full rank left annihilator $G^\perp (x)$ of $G(x)$. 	
Find mappings $F_d: \R^n \to  \R^{n \times n}$ and $H_d: \R^n \to \R$ such that for all $x \in \Omega$
	\begin{subequations}
	\begin{align}
		& F_d(x) + F_d^T(x) \preceq 0   \label{eq:JR} \\
		&  G^\perp (x)  F_d(x) \nabla H_d(x) = G^\perp (x) f(x)  \label{eq:mat}  \\
		& \nabla H_d(x^*) = 0, ~\frac{\partial^2 H_d}{\partial x^2} (x^*) \succ 0  \label{eq:H}.
	\end{align}
	\end{subequations}
\end{Problem}

The PDE~\eqref{eq:mat} is the so-called \emph{matching equation} of IDA-PBC. The minimum condition~\eqref{eq:mini} is ensured by~\eqref{eq:H}. 
If Problem~\ref{prb:ida} is solved, then the state feedback
\be \label{eq:u}
    u = G^\dagger(x) (F_d(x) \nabla H_d(x)  - f(x) ) 
\ee 
transforms~\eqref{eq:sys} into~\eqref{eq:pch}, and renders $x^*$ stable. Here, $G^\dagger(x):=(G^T(x) G(x))^{-1} G^T(x) $ is the pseudoinverse of $G(x)$. By \emph{LaSalle's} invariance principle,  the asymptotically stability of $x^*$ is achieved 
if the largest invariant set under the dynamics of \eqref{eq:pch} contained in 
$
		\{x \in \Omega ~|~ \nabla ^T H_d(x) R_d(x) \nabla H_d(x) = 0  \}
$
equals $\{ x^*\}$. 

Note that the main difficulty of~Problem~\ref{prb:ida} is to solve the matching equation~\eqref{eq:mat}. If $F_d(x)$ is a \emph{priori} fixed such that~\eqref{eq:JR} holds, a necessary and sufficient condition for the solvability of~\eqref{eq:mat} is given in~\cite{cheng2005feedback}. 

\begin{Proposition} \cite[Prop. 3]{cheng2005feedback} \label{prop:solva}
Define 	$W(x) :=  G^\perp(x) F_d(x)$ and $s(x):=G^\perp (x) f(x)$.
Let $w_i(x)$ and $s_i(x)$, $i = 1, \dots, n-m$, denote the $i$th column vector of $ W^T(x) $ and the $i$th entry of $s(x)$, respectively. Assume that the following two distributions
	\begin{align}
		\varDelta (x) :=& \spanop\{w_1, \dots, w_{n-m} \}   \label{eq:dstr} \\
		\tilde \varDelta(x) := & \spanop \left \{\begin{bmatrix} w_1 \\s_1 \end{bmatrix} , \dots,  \begin{bmatrix} w_{n-m} \\s_{n-m} \end{bmatrix} \right \}  \label{eq:dstr1}
	\end{align}
are regular \updt{for all $x \in \Omega$}. Let $ \inv \varDelta$ and $ \inv \tilde \varDelta $ denote the involutive closures of $\varDelta(x)$ and $\tilde \varDelta(x)$, respectively. Then the matching equation~\eqref{eq:mat}, i.e., $W(x)\nabla H_d(x) = s(x)$, admits a solution $H_d(x)$ iff \updt{there exists some $d \in \{n-m, \dots, n\}$ such that}
$\dim( \inv \varDelta) = \dim ( \inv \tilde \varDelta) = d$ \updt{for all $x \in \Omega$}.
\end{Proposition}

Prop.~\ref{prop:solva} only ensures the existence of an energy function $H_d(x)$. However, such a $H_d(x)$ does not necessarily satisfy~\eqref{eq:H}.  Ref.~\cite{kotyczka2013local} shows that $H_d(x)$ has certain degrees of freedom to design to achieve~\eqref{eq:H}, that is, the so-called \emph{energy shaping} procedure. To elucidate this, consider the homogeneous part of the matching equation~\eqref{eq:mat}, i.e.,
\be \label{eq:homopde}
G^{\perp}(x) F_d(x) \nabla H_d (x) = 0.
\ee 
Note that~\eqref{eq:homopde} always has a solution since the condition $\dim( \inv \varDelta) = \dim ( \inv \tilde \varDelta) = d$ in Prop.~\ref{prop:solva} holds naturally with $s(x) = 0$. As a fact, $H_d(x) = 0$ is a trivial solution.

\begin{Proposition} \label{prop:char} \cite[Prop. 3]{kotyczka2013local}
Assume that the distribution $\varDelta(x)$ defined in~\eqref{eq:dstr} is regular \updt{for all $x \in \Omega$}. Let $d := \dim( \inv \varDelta)$. If $d< n$, then~\eqref{eq:homopde} admits $n-d$ independent solutions on \updt{$\Omega$}, that is, there exist $\xi_i: \R^n \to \R$, $i = 1, \dots, n-d$, such that
\be \label{eq:homopde1}
G^{\perp}(x) F_d(x) \nabla \xi_i (x) = 0.
\ee
In particular, if~$\varDelta(x)$ is involutive~\updt{for all $x \in \Omega$}, i.e., $d = n-m$,  then exist $m$ such mappings $\xi_i(x)$.
\end{Proposition}

In line with~\cite[Def. 2]{kotyczka2013local}, we call the mappings $\xi_1, \dots, \xi_{n-d}$ \emph{characteristic coordinates}.
Define $\xi(x) := \begin{bmatrix} \xi_1(x) & \cdots & \xi_{n-d}(x) \end{bmatrix}^T$.
By~\eqref{eq:homopde1}, the solution to~\eqref{eq:mat} can be decomposed as
\be \label{eq:Hdecom}
H_d(x) = \Psi(x) + \Phi(\xi(x))
\ee
where $\Psi: \R^n \to \R$ is a particular solution to~\eqref{eq:mat}, and  $\Phi:\R^{n-d} \to \R$ is an arbitrary smooth function to be designed in the energy shaping procedure. This implies that the characteristic coordinates specify the energy shapeability of IDA-PBC.

\section{Main Results} \label{sec:main}
\updt{Following a similar terminology used in~\cite{kotyczka2013local}, we formally define \emph{maximum energy shapeability} as follows. }

\begin{Definition} \label{def:mes} \updt{Let $G^{\perp}(x)$ be a full-rank left annihilator of $G(x)$ and $F_d:\mathbb{R}^n\to\mathbb{R}^{n\times n}$ satisfy
		$
		F_d(x)+F_d^T(x)\preceq 0
		$ for all $ x\in\Omega$. }
	The system~\eqref{eq:sys} is said to have \emph{maximum energy shapeability} \updt{w.r.t. $G^{\perp}(x)$ and $F_d(x)$} if 
	\begin{enumerate}[a)]
		\item The PDE~\eqref{eq:mat} admits a solution $H_d(x)$ \updt{on $\Omega$};
		\item There exist $m$ characteristic coordinates, that is, the PDE~\eqref{eq:homopde1} admits $m$ independent solutions $\xi_i(x)$ \updt{on $\Omega$}, $i = 1, \dots, m.$
	\end{enumerate}
\end{Definition} 

\updt{There exist at most $m$ characteristic coordinates since $d \geq n-m$. For the system~\eqref{eq:sys} with an $m$-dimensional control input,  the maximum energy shapeability corresponds to the case 
in which $H_d(x)$ retains the largest possible number of independent degrees of freedom for energy shaping allowed by the actuation dimension.}

If the system~\eqref{eq:sys} has maximum energy shapeability, we can introduce a \updt{(local)} change of coordinates $z = \tau(x)$, $\tau: \R^n \to \R^n$. Specifically, 
\be  \label{eq:taux}
\tau(x) = \begin{bmatrix} \xi(x) \\ \eta(x)  \end{bmatrix}
\ee
where each entry of $\xi: \R^n \to \R^m$ is a characteristic coordinate. The mapping $\eta: \R^n \to \R^{n-m}$ can be defined arbitrarily such that \updt{$\frac{\partial \tau}{\partial x}(x^*)$} is nonsingular, and this ensures that $\tau(x)$ is a \updt{local diffeomorphism in a neighborhood of $x^*$}. By the \emph{Inverse Function Theorem} \cite[p. 471]{isidori-book}, the inverse coordinate transformation \updt{$\bar \tau: \R^n \to \R^n$ exists locally such that $x = \bar \tau(z)$}. For a desired equilibrium point $x \in \mathcal{E}$, define
\begin{align*}
	z^*  := \tau(x^*) ,~
	\tilde z  := z - z^*. 
\end{align*}

The next result directly follows from~\cite[Prop. 4]{kotyczka2013local}, which shows that the matching equation is equivalent to a reduced order equation under the change of coordinates~\eqref{eq:taux}. We provide a proof here to facilitate the subsequent analysis.

\begin{Lemma} \label{lem:nabeta} 
Assume the system~\eqref{eq:sys} has maximum energy shapeability and	
introduce the change of coordinates $z= \tau(x)$ as defined in~\eqref{eq:taux}.
Then the PDE~\eqref{eq:mat} reduces to
\be  \label{eq:nabeta} 
\nabla_\eta \bar H_d(z) = \left ( \left (G^\perp F_d  \frac{\partial^T \eta }{\partial x} \right)^{-1} G^\perp  f \right ) \circ   \updt{\bar \tau(z)}
\ee
where $\bar H_d(z) :=  H_d( \updt{\bar \tau(z)})$. 

\end{Lemma}

\begin{proof}
	Applying the chain rule to $\bar H_d(z)$ yields 
	\be \label{eq:chain}
	\frac{\partial \bar H_d}{\partial z}(z) = \frac{\partial  H_d}{\partial x}( \updt{\bar \tau(z)}) \frac{\partial  \updt{\bar \tau}}{\partial z}(z).
	\ee 
	Since $\tau(\updt{\bar \tau(z)}) = z$, we have 
	$
	\frac{\partial  \updt{\bar \tau}}{\partial z}(z) = \left ( \frac{\partial \tau }{\partial x}( \updt{\bar \tau(z)}) \right )^{-1}.
	$
Hence, 
	\be \label{eq:nabHxz}
	\nabla_x H_d( \updt{\bar \tau(z)} ) = \frac{\partial^T \tau }{\partial x}( \updt{\bar \tau(z)}) \nabla_z \bar H_d(z). 
	\ee
Then \eqref{eq:mat} is transformed into
	\be \label{eq:matz}
	\left (G^\perp F_d  \frac{\partial^T \tau }{\partial x} \right)\circ  \updt{\bar \tau(z)}  \nabla_z \bar H_d(z) =  ( G^\perp  f) \circ  \updt{\bar \tau(z)}.
	\ee
By~\eqref{eq:homopde1}, Eq.~\eqref{eq:matz} reduces to
	$ 
	 (G^\perp F_d  \frac{\partial^T \eta }{\partial x}  )\circ  \updt{\bar \tau(z)}  \nabla_\eta \bar H_d(z) =  ( G^\perp  f) \circ  \updt{\bar \tau(z)}
	$.
Note that $G^\perp(x) F_d(x) \frac{\partial^T \eta }{\partial x}(x)$ is a square matrix in $\R^{(n-m) \times (n-m)}$. 
By Prop.~\ref{prop:char}, the maximum energy shapeability requires that the distribution~\eqref{eq:dstr} in Prop.~\ref{prop:solva} is involutive, and thus $G^\perp(x) F_d(x)$ is of full rank. Therefore, the nonsingularity of  $\frac{\partial \tau }{\partial x}$ implies that $G^\perp(x) F_d(x) \frac{\partial^T \eta }{\partial x}(x)$ is invertible, and this yields~\eqref{eq:nabeta}. 
\end{proof}

\begin{Remark} \label{re:Heta0}
Since $G^\perp(x^*) f(x^*) = 0$ for any $x^* \in \mathcal{E}$, Eq.~\eqref{eq:nabeta} implies that $\nabla_\eta \bar H_d(z^*) = 0$ holds for any choice of $\eta(x)$. 
\end{Remark}

\begin{Remark}
Since~\cite{kotyczka2013local} focuses on the parameter tuning of the IDA-PBC method, the change of coordinates $\tau(x)$ defined in~\eqref{eq:taux}, and in particular the mapping $\eta(x)$, is not required to be explicitly defined. In this work, however, an explicit definition of $\eta(x)$ is necessary to establish the following results.
\end{Remark}

By~\eqref{eq:nabHxz}, the control law induced by IDA-PBC, i.e., \eqref{eq:u}, can be rewritten as 
\begin{align} 
	u =&  G^\dagger(x) \Big (-f(x)  + F_d(x) \frac{\partial^T \xi}{\partial x}(x) \nabla_\xi \bar H_d(z)  \nonumber \\
	& + F_d(x) \frac{\partial^T \eta}{\partial x}(x) \nabla_\eta \bar H_d(z) \Big ). \label{eq:uinz}
\end{align}
Lemma~\ref{lem:nabeta} shows that $\nabla_\eta \bar H_d(z)$ is known once the characteristic coordinates $\xi_1, \dots, \xi_m$ is obtained by solving~\eqref{eq:homopde1}. Therefore, the design of~\eqref{eq:uinz} now only requires the knowledge of $\nabla_\xi \bar H_d(z)$. The subsequent analysis shows that $\nabla_\xi \bar H_d(z)$ can be constructed by using $\nabla_\eta \bar H_d(z)$. 
For simplicity, we denote $\rho(z) := \nabla_\eta \bar H_d(z)$, i.e.,
\be \label{eq:defrho}
\rho(z) := \left ( \left (G^\perp F_d  \frac{\partial^T \eta }{\partial x} \right)^{-1} G^\perp  f \right ) \circ  \updt{\bar \tau(z)}.
\ee

\begin{Theorem} \label{thm:main}
If the system~\eqref{eq:sys} achieves maximum energy shapeability \updt{for some $G^{\perp}(x)$ and $F_d(x)$}, then there exists a mapping $M: \Omega \to \R^{m \times m}$ such that
\be \label{eq:Msym}
M(z) = M^T(z)
\ee
and for all $i,j \in \{1, \dots, m\}$
\begin{align}
\frac{\partial M_{ij}}{\partial \xi_k} &= \frac{\partial M_{ik}}{\partial \xi_j}, \quad  k \in \{1, \dots, m \}  \label{eq:Mxi} \\
\frac{\partial M_{ij}}{\partial \eta_k} &= \frac{\partial^2 \rho_k}{\partial \xi_i \partial \xi_j}, \quad k \in \{1, \dots, n-m  \}  \label{eq:Meta}
\end{align}
where $M_{ij}$ denotes the $ij$th entry of $M$, and $\xi_i, \eta_i, \rho_i$ is the $i$th entry of $\xi, \eta, \rho$, respectively. 
Furthermore, there exist mappings $\beta: \R^n \to \R^m$ , $\theta: \R^n \to \R$ such that 
\begin{align}
	\frac{\partial \beta}{\partial z} &= \begin{bmatrix}
		M(z) & \frac{\partial^T \rho}{\partial \xi}( z)
	\end{bmatrix}   \label{eq:nabbeta} \\
	\frac{\partial \theta}{\partial z} &= \begin{bmatrix}
		\beta^T(z) & \rho^T(z) \label{eq:nabtheta}
	\end{bmatrix}
\end{align}
and $\bar H_d(z) = \theta(z)$ is a solution of the PDE~\eqref{eq:nabeta}. 
\end{Theorem}

\begin{proof}
 The maximum energy shapeability ensures that~\eqref{eq:nabeta} in Lemma~\ref{lem:nabeta} admits a solution~$\bar H_d(z)$. Let $M(z) : = \frac{\partial^2 \bar H_d}{\partial \xi^2}$. We claim that such a construction of $M(z)$ satisfies~\eqref{eq:Msym}-\eqref{eq:Meta}. 
	
According to \emph{Poincar\'e's Lemma} and using~\eqref{eq:nabeta}, the Hessian matrix of $\bar H_d(z)$, i.e.,
	$
	\frac{\partial^2 \bar H_d}{\partial z^2} = 
	\begin{bmatrix}
 	 M  & \frac{\partial^2 \bar H_d }{\partial \xi \partial \eta} \\
		\frac{\partial \rho}{\partial \xi}  & \frac{\partial \rho }{\partial \eta}
	\end{bmatrix}
	$ 
is symmetric. This implies that~\eqref{eq:Msym} holds and 
$ 
\frac{\partial^2 \bar H_d }{\partial \xi \partial \eta}(z) = \frac{\partial^T \rho}{\partial \xi}(z).
$
Therefore, $\frac{\partial^2 \bar H_d(z)}{\partial \xi \partial z}    = Q(z)$ with 
$
Q(z):= \begin{bmatrix} M(z) & \frac{\partial^T \rho}{\partial \xi}(z) \end{bmatrix}.
$
By the \emph{mixed-partial symmetry requirement} for higher-order derivatives (see e.g., \cite[Prop. 2.4.14]{abraham2012manifolds}), we have 
$
\frac{\partial Q_{ij}}{\partial z_k} = \frac{\partial Q_{ik}}{\partial z_j}, 
$
which is equivalent to~\eqref{eq:Mxi} and~\eqref{eq:Meta}. That is, there exists a $M(z)$ satisfying~\eqref{eq:Msym}-\eqref{eq:Meta}. 

The existence of the mappings $\beta(z)$ and $\theta(z)$ directly follows from the existence of~$M(z)$. Indeed, Eq.~\eqref{eq:Mxi} and~\eqref{eq:Meta} correspond to an integrability condition for the row vector fields of $Q(z)$, i.e., the right-hand side of~\eqref{eq:nabbeta}. Eq.~\eqref{eq:Msym} ensures that the right-hand side of~\eqref{eq:nabtheta} is integrable. Clearly, $\theta(z)$ is a solution of~\eqref{eq:nabeta} since $\frac{\partial \theta}{\partial \eta} = \rho^T(z)$.
\end{proof}

Lemma~\ref{lem:nabeta} implies that if maximum energy shapeability is ensured, the original matching equation~\eqref{eq:mat} is equivalent to the PDEs~\eqref{eq:homopde1} and~\eqref{eq:nabeta}. \updt{In the standard IDA-PBC, recall that the homogeneous PDE~\eqref{eq:homopde1} must be solved as part of the energy shaping procedure, namely, to construct~$\Phi(\xi(x))$ in~\eqref{eq:Hdecom}. For the remaining PDE~\eqref{eq:nabeta}, 
Theorem~\ref{thm:main} shows that~\eqref{eq:nabeta} is equivalent to the PDEs~\eqref{eq:Mxi}-\eqref{eq:Meta} subject to~\eqref{eq:Msym}. The next example shows that~\eqref{eq:Msym}-\eqref{eq:Meta} admit a constructive solution for a single input system.}

\updt{
\begin{Example} \label{exm:sicase}
Consider the system~\eqref{eq:sys} with~$m =1$, then both $\xi(x)$ and $M(z)$ are scalar-valued mappings. In this case, Eqs. \eqref{eq:Msym}-\eqref{eq:Meta} reduce to  
\be  \label{eq:sicase}
\frac{\partial M}{\partial \eta_k} = \frac{\partial^2 \rho_k}{\partial \xi^2}, ~k \in \{1, \dots, n-1 \}. 
\ee
Let $\ell_k(z):= \frac{\partial^2 \rho_k}{\partial \xi^2}(z)$. Note that~\eqref{eq:sicase} admits an explicit solution by successive integration. For example, if $n =3$, then 
\begin{align*}
	M(z) = \int \ell_1(z) \mathrm{d} \eta_1 + \int \tilde \ell_2(z) \mathrm{d} \eta_2 + h(\xi) 
\end{align*}
where $\tilde \ell_2(z) := \ell_2(z) - \frac{\partial }{\partial \eta_2}\int \ell_1(z) \rm d \eta_1$ and $h: \R \to \R$ is an arbitrary smooth mapping.
\end{Example}
}

The next result follows immediately from Theorem~\ref{thm:main}.

\begin{Corollary} \label{coro:int}
Let $M: \Omega \to \R^{m \times m}$ satisfy~\eqref{eq:Msym}-\eqref{eq:Meta}, then 
\begin{equation} \label{eq:Hdzint}
	\bar H_d(z) = \left (\int_0^1 \begin{bmatrix} \beta^T(z^* + \lambda \tilde z) &  \rho^T(z^* + \lambda \tilde z) \end{bmatrix} \mathrm{d} \lambda \right ) \tilde z
\end{equation}
with 
\begin{equation} \label{eq:beta}
	\beta  =  
	\left( \int_0^1 
	\begin{bmatrix}
		M(z^* + \lambda \tilde z) & \frac{\partial^T \rho}{\partial \xi}(z^* + \lambda \tilde z)
	\end{bmatrix}  \mathrm{d} \lambda \right) \tilde z
\end{equation}
is a solution of the PDE~\eqref{eq:nabeta}.
\end{Corollary}

\begin{proof}
Recall the Mean Value Theorem for vector-valued functions (see e.g., \cite[Porp. 2.4.7]{abraham2012manifolds}), that is, for any $C^1$ mapping $v: \R^n \to \R^m$, 
\be \label{eq:mvthm}
v(a) - v(b) = \left ( \int_0^1 \frac{\partial v}{\partial x}( b + \lambda (a -b) ) \mathrm{d} \lambda \right) (a - b). 
\ee
Theorem~\ref{thm:main} implies that $ \beta(z) = \nabla_\xi \bar H_d(z)$ for some solution $H_d(z)$ of~\eqref{eq:nabeta}. Note that \eqref{eq:nabeta} is independent of $\nabla_\xi \bar H_d(z)$. Hence, the value of  $\nabla_\xi \bar H_d(z)$, or equivalently, $\beta(z)$, at a fixed point can be arbitrarily assigned, and we can set $\beta(z^*) = 0$. Applying~\eqref{eq:mvthm} to~\eqref{eq:nabbeta} yields~\eqref{eq:beta}. Then we have $\frac{\partial \bar H_d}{\partial z}(z) = \begin{bmatrix} \beta^T(z) & \rho^T(z) \end{bmatrix}$. Since $\bar H_d(z)+ c$ with $c$ an arbitrarily constant is also a solution of \eqref{eq:nabeta}, we can set $\bar H_d(z^*) = 0$, and using~\eqref{eq:mvthm} again yields~\eqref{eq:Hdzint}.
\end{proof}

Corollary~\ref{coro:int} shows that there exists a $\bar H_d(z)$ such that $\nabla_\xi \bar H_d(z^*) = 0$, together with Remark~\ref{re:Heta0}, we have $\nabla_z \bar H_d(z^*) = 0$. Recall~\eqref{eq:nabHxz}, $\nabla_x H_d(x^*) = 0$ in~\eqref{eq:H} is ensured. The next result shows that $\frac{\partial^2 H_d}{\partial x^2} (x^*) \succ 0 $ in~\eqref{eq:H} can also be satisfied due to that the matrix  $M(z)$ specified in Theorem~\ref{thm:main} has additional degrees of freedom to design, and this ensures that the equilibrium $x^*$ is stable as discussed in Section~\ref{sec:pre}.

\begin{Theorem} \label{thm:m12}
Assume that the system~\eqref{eq:sys} has maximum energy shapeability \updt{for some $G^{\perp}(x)$ and $F_d(x)$}. If the mapping $\rho(z)$ defined in~\eqref{eq:defrho} satisfies
\be \label{eq:eta2}
\frac{\partial \rho}{\partial \eta}(z^*) \succ 0
\ee
then the matching equation~\eqref{eq:mat} admits a solution $H_d(x)$ such that $\frac{\partial^2 H_d}{\partial x^2} (x^*) \succ 0 $. 
\end{Theorem}

\begin{proof}
Let~$\bar H_d(z)$ be a solution of~\eqref{eq:nabeta} in Lemma~\ref{lem:nabeta}. By Theorem~\ref{thm:main}, the Hessian matrix of~$\bar H_d(z)$ takes the form
\be \label{eq:Hess1}
\frac{\partial^2 \bar H_d}{\partial z^2}(z) = 
\begin{bmatrix}
	M(z)  &  \frac{\partial^T \rho}{\partial \xi}(z) \\
	\frac{\partial \rho}{\partial \xi}(z)  & \frac{\partial \rho }{\partial \eta}(z)
\end{bmatrix}
\ee
where the matrix $M(z)$ satisfies~\eqref{eq:Msym}-\eqref{eq:Meta}. It is straightforward to show that $M(z)$ can be decomposed as
\be \label{eq:M1M2}
M(z) = M_1(z) + M_2(\xi)
\ee
where $M_1(z)$ satisfies~\eqref{eq:Msym}-\eqref{eq:Meta}, and $M_2(\xi)$ satisfies~\eqref{eq:Msym}-\eqref{eq:Mxi}. Note that $M_2(\xi)$ can be arbitrarily designed. For example, any constant symmetric $M_2$ trivially ensures that~\eqref{eq:Msym}-\eqref{eq:Mxi} hold. For simplicity, the subsequent analysis considers a constant $M_2$. Let $B:= \frac{\partial \rho}{\partial \eta}(z^*)$, $C:= \frac{\partial \rho}{\partial \xi}(z^*)$, then
\begin{equation*} 
\frac{\partial^2 \bar H_d}{\partial z^2}(z^*) = 
\begin{bmatrix}
M_1(z^*)+M_2  &  C^T \\
C  & B
\end{bmatrix}.
\end{equation*}
By~\eqref{eq:eta2}, i.e., $B\succ 0$, the Schur complement of $B$ exists and is given by
$
D := M_1(z^*)+M_2 - C^T B^{-1} C 
$.
Note that $C^T B^{-1} C \succeq 0$ and $M_1(z^*) = M_1^T(z^*)$. Recall that for any symmetric matrix $A \in \R^{m \times m}$ and $a \in \R^m$, $ \lambda_\text{min}(A) |a|^2 \leq a^T A a \leq  \lambda_\text{max}(A) |a|^2$. Here, $\lambda_\text{min}(A), \lambda_\text{max}(A)$ denote the smallest and largest eigenvalues of $A$, respectively. Therefore, by picking $M_2$ such that
\be \label{eq:M2lam}
\lambda_\text{min}(M_2) > \lambda_\text{max}( C^T B^{-1} C)- \lambda_\text{min}(M_1(z^*))  
\ee
it is ensured that $D \succ 0$. Together with $B \succ 0$, we have $\frac{\partial^2 \bar H_d}{\partial z^2}(z^*) \succ 0$.
Since $\nabla_z \bar H_d(z^*) = 0$ by Remark~\ref{re:Heta0} and Corollary~\ref{coro:int}, Eq.~\eqref{eq:nabHxz} leads to 
$
\frac{\partial^2 H_d}{\partial x^2} (x^*) = \frac{\partial^T \tau }{\partial x}(x^*) \frac{\partial^2 \bar H_d}{\partial z^2} (z^*) \frac{\partial \tau }{\partial x}(x^*).
$
The nonsingularity of $\frac{\partial^T \tau }{\partial x}(x^*)$ ensures that $\frac{\partial^2 H_d}{\partial x^2} (x^*) \succ 0$. 
\end{proof}

\updt{
\begin{Remark}
	In~\eqref{eq:M1M2}, the mappings $M_1(z)$ and $M_2(\xi)$ correspond to the particular solution~$\Psi(x)$ and the homogeneous solution~$\Phi(\xi)$ in~\eqref{eq:Hdecom}, respectively.
\end{Remark}
}

\section{Sufficient Conditions for Maximum Energy Shapeability} \label{sec:suff}

Based on Props.~\ref{prop:solva}-\ref{prop:char}, this section specifies several sufficient conditions for maximum energy shapeability. 

\begin{Theorem} \label{thm:underone}
	If the system~\eqref{eq:sys} with $x \in \R^n$ and $u \in \R^m$ satisfies $m = n-1$, then \updt{there exists a pair of $G^{\perp}(x)$ and $F_d(x)$} such that \eqref{eq:sys} achieves maximum energy shapeability. 
\end{Theorem}

\begin{proof}
	For the case $m = n-1$, $G^\perp(x)$ reduces to a row vector. We can select the matrices $G^\perp(x)$ and $F_d(x)$ such that $G^\perp(x) F_d(x)$ is nonzero. In this case, the distributions $\varDelta(x)$ and $\tilde \varDelta(x)$ defined in Prop.~\ref{prop:solva} are both regular and involutive since they contain only one nonzero vector field, that is, $\dim(\inv \varDelta) = \dim ( \inv \tilde \varDelta) = 1$. The maximum energy shapeability follows from Prop.~\ref{prop:solva} and Prop.~\ref{prop:char}. 
\end{proof}

\updt{
\begin{Remark}
	Together with Example~\ref{exm:sicase}, Theorem~\ref{thm:underone} implies that IDA-PBC can be readily applied to a $2$-dimensional affine nonlinear system with a single input. In this case, the conditions~\eqref{eq:Msym}-\eqref{eq:Meta} in Theorem~\ref{thm:main} reduce to 
	\be \label{eq:Msca}
	\frac{\partial M}{\partial \eta}(z) =  \frac{\partial^2 \rho}{\partial \xi^2}(z)
	\ee
	where $z \in \R^2$, $\xi \in \R$, $\eta \in \R$, and $\rho: \R^2 \to \R$ are defined as in Section~\ref{sec:main}. Note that~\eqref{eq:Msca} is an ODE in~$\eta$ that   admits an explicit solution by direct integration. 
\end{Remark}
}

\begin{Remark} \label{re:underone}
	\updt{IDA-PBC has been widely applied to mechanical systems described as the pH form (see e.g., \cite{acosta2005interconnection}), i.e., 
	\be
		\begin{bmatrix}\dot q \\ \dot p\end{bmatrix}
		=
		\begin{bmatrix}0 & I_n \\ -I_n & 0\end{bmatrix}
		\nabla H(q,p)
		+
		\begin{bmatrix}0 \\ G(q) \end{bmatrix}u
		\label{eq:mechph}
	\ee
where $q, p\in\mathbb{R}^n$, and $u \in \R^m$ are the generalized position, the
generalized momenta, and the control input, respectively. The Hamiltonian $H: \R^{2n} \to \R$ is 
\begin{equation*}
	H(q,p) = \frac12 p^T L^{-1}(q) p + V(q)
\end{equation*}
where $L = L^T \succ 0$ is the inertia matrix, and $V:\R^n \to \R$ is the potential energy.
 For~\eqref{eq:mechph},} the matching equation can be separated into two elements, i.e., the \emph{kinetic energy} PDEs which depend on both $p$ and $q$, and the \emph{potential energy} PDEs which only depend on $q$.
 \updt{Some researchers~\cite{acosta2005interconnection, arpenti2022constructive, ryalat2016simplified} consider a particular case with $m = n-1$, that is, of underactuation degree one, and show that} the matching equation can be explicitly solved. Specifically, the kinetic energy PDEs can be transformed into algebraic ones by fixing the desired inertia matrix, then the remaining potential energy PDEs have the same structure as the matching equation in the case of Theorem~\ref{thm:underone}. \updt{Therefore, Theorem~\ref{thm:underone} provides a theoretical basis for the existence of an explicit solution to the matching equation arising in~\eqref{eq:mechph} with~$m= n-1$.}
\end{Remark}

\begin{Theorem} \label{thm:poin}
The system~\eqref{eq:sys} achieves maximum energy shapeability if there exists matrices $G^\perp(x)$ and $F_d(x)$ such that 
\begin{enumerate}[a)]
	\item \label{cond:mes3} the matrix $F_d(x)$ is nonsingular, and satisfies
	\be \label{eq:condmes3}
	\frac{\partial F_d^{-1}g_i }{\partial x} (x) =  \frac{\partial^T F_d^{-1}g_i }{\partial x} (x), \quad i = 1,\dots, m
	\ee
	where $g_i(x)$ is the $i$th column of $G(x)$;
	\item \label{cond:mes4} the distribution $\tilde \varDelta(x)$ defined in Prop.~\ref{prop:solva} satisfies $$\dim ( \inv \tilde \varDelta) = n-m.$$
\end{enumerate}
\end{Theorem}

\begin{proof}
	According to \emph{Poincar\'e's Lemma}, Condition~\eqref{cond:mes3} implies that $F_d^{-1}(x)g_i(x)$, $i = 1, \dots, m$, are gradient vector fields, that is, there exist $\xi_i: \R^n \to \R$, $i = 1, \dots, m$, such that 
	$$\nabla \xi_i(x) = F_d^{-1}(x)g_i(x).$$
	These mappings $\xi_i(x)$ are independent since both $F_d(x)$ and $G(x)$ have full rank.
	Note that 
	$G^{\perp}(x)F_d(x) \nabla \xi_i(x) = 0$ in this case, that is, the PDE \eqref{eq:homopde1} admits $m$ independent solutions. By \emph{Frobenius' Theorem} (see e.g., \cite[p. 23]{isidori-book}), this implies that the distribution $\varDelta(x)$ defined in Prop.~\ref{prop:solva} is involutive, that is, $\dim( \inv \varDelta) = n-m$. 
	Together with Condition~\eqref{cond:mes4}, Prop.~\ref{prop:solva} ensures that the matching equation~\eqref{eq:mat} is solvable in this case. This completes the proof upon recalling Definition~\ref{def:mes}.
\end{proof}

\updt{
\begin{Remark}
Theorem~\ref{thm:poin} extends the result in~\cite{borja2016constructive} 
where a condition similar to~\eqref{eq:condmes3} is exploited to develop a constructive IDA-PBC design for a class of pH-systems.
\end{Remark}
}

\begin{Theorem} \label{thm:cnst}
The system~\eqref{eq:sys} achieves maximum energy shapeability if there exists matrices $G^\perp(x)$ and $F_d(x)$ such that 
	\begin{enumerate}[a)]
		\item \label{cond:mes1} the matrix $G^\perp(x) F_d(x)$ is constant and has full rank;
		\item \label{cond:mes2} the mappings $w_i(x)$ and $s_i(x)$ defined in Prop.~\ref{prop:solva} satisfy
		\begin{equation*} 
		L_{w_i} s_j(x) - L_{w_j} s_i(x) = 0, \text{ for all }i, j =1, \dots, n-m. 
		\end{equation*}
	\end{enumerate}  
\end{Theorem}

\updt{The proof of Thm.~\ref{thm:cnst} follows from Prop.~\ref{prop:solva} and~\cite[Prop. 5]{kotyczka2013local}, and is therefore omitted.}

\section{Application Examples} \label{sec:simu}
\updt{Our results implies that the IDA-PBC method, i.e., Problem~\ref{prb:ida}, can be conducted by the systematic procedure as shown in~Fig.~\ref{fig:flowchart} if maximum energy shapeability is achieved. The parameter tuning in $G^\perp(x)$ and $F_d(x)$ is ad-hoc in general. Furthermore, the design of $G^\perp(x)$ and $F_d(x)$ involves free parameters beyond those required to achieve maximum energy shapeability, which can be further chosen to ensure $\frac{\partial^2 \bar H_d}{\partial \eta^2}(z^*) \succ 0$, i.e.,~\eqref{eq:eta2} in Theorem~\ref{thm:m12}.}

\begin{figure}[t]
	\centering
	\begin{tikzpicture}[
font=\footnotesize,
node distance=4.2mm,
flow/.style={
    -{Latex[length=1.8mm,width=1.2mm]},
    line width=0.45pt
},
box/.style={
    rectangle,
    rounded corners=1.5pt,
    draw=black,
    fill=blue!5,
    line width=0.45pt,
    align=left,
    text width=0.78\columnwidth,
    inner xsep=4pt,
    inner ysep=3.5pt
},
decision/.style={
    diamond,
    draw=black,
    fill=blue!5,
    line width=0.45pt,
    aspect=5.3,
    align=center,
    text width=0.38\columnwidth,
    inner sep=0.1pt
}
]

\node[box] (s1) {
\textbf{Step 1:}
Select \(G^{\perp}(x)\) and \(F_d(x)\) 
};


\node[decision, below=of s1] (d1) {
Maximum energy\\
shapeability achieved?
};


\node[box, below=of d1] (s2) {
\textbf{Step 2:}
Solve
$
G^{\perp}(x)F_d(x)\nabla\xi_i(x)=0
$
to obtain the characteristic coordinates $\xi_i$, $i = 1, \dots, m$.
};


\node[box, below=of s2] (s3) {
\textbf{Step 3:}
Choose $\eta(x)$ and design a change of coordinate $z =\tau(x)$ as in~\eqref{eq:taux}. Compute 
\(\nabla_{\eta}\bar H_d(z)\) using Lemma~\ref{lem:nabeta}.
};


\node[decision, below=of s3] (d3) {
Verify
\(\displaystyle
\frac{\partial^2\bar H_d}
{\partial\eta^2}(z^*)\succ0
\)
};


\node[box, below=of d3] (s4) {
\textbf{Step 4:}
Solve~\eqref{eq:Msym}-\eqref{eq:Meta}  to obtain \(M_1(z)\).
Choose a constant symmetric matrix \(M_2\)
satisfying~\eqref{eq:M2lam}.
};


\node[box, below=of s4] (s5) {
\textbf{Step 5:}
Let
$
M(z)=M_1(z)+M_2.
$
Compute \(\nabla_{\xi}\bar H_d(z)\) from~\eqref{eq:beta}.
Using~\eqref{eq:uinz} to obtain the control law.
};


\draw[flow]
(s1.south) -- (d1.north);

\draw[flow]
(d1.south) --
node[right, xshift=1mm]{\textbf{Yes}}
(s2.north);

\draw[flow]
(s2.south) -- (s3.north);

\draw[flow]
(s3.south) -- (d3.north);

\draw[flow]
(d3.south) --
node[right, xshift=1mm]{\textbf{Yes}}
(s4.north);

\draw[flow]
(s4.south) -- (s5.north);


\coordinate (loop1top)
at ($(s1.west)+(-8mm,0)$);

\coordinate (loop1bottom)
at (loop1top |- d1.west);

\draw[flow]
(d1.west)
--
(loop1bottom)
node[midway, below=1mm]{\textbf{No}}
--
(loop1top)
--
(s1.west);


\coordinate (loop3top)
at ($(s3.west)+(-8mm,0)$);

\coordinate (loop3bottom)
at (loop3top |- d3.west);

\draw[flow]
(d3.west)
--
(loop3bottom)
node[midway, below=1mm]{\textbf{No}}
--
(loop3top)
--
(s1.west);

\end{tikzpicture}
	\caption{Flow chart of the proposed procedure for IDA-PBC.}
	\label{fig:flowchart}
\end{figure}
\vspace{3pt}

\subsection{Magnetic Levitation System} \label{sec:simu_mls}
The magnetic levitation system \cite{ortega2001putting} consists of an iron ball and an electromagnet which creates a vertical magnetic field. Let $m$ denote the mass of the ball, $\lambda$ the flux of the electromagnet, $u$ the voltage across the coil, and $a$ the gravitational acceleration. Let $\theta $ denote the difference between the position of the center of the ball and its nominal position. Define the state variables as $y:= \begin{bmatrix} \lambda & \theta & m \dot \theta \end{bmatrix}^T$, then the dynamics of the magnetic levitation system are
\begin{align}
\dot y_1 =& -\frac{\gamma}{k}(c - y_2) y_1 + u \nonumber  \\ 
\dot y_2 =& \frac{y_3}{m}  \label{eq:mls}  \\ 
\dot y_3 =& -ma + \frac{y_1^2}{2k}   \nonumber
\end{align}
where $k$ and $c$ are positive constants. The value of $k$ depends on the number of coil turns. It is assumed that $y_2 < c$ and the ball touches the electromagnet at $y_2 = c$.
In the subsequent analysis, the model parameters are selected as:  
$
	 \gamma = \SI{2.52}{\ohm}, k = \SI{6.4042e-5}{\newton \metre / \ampere},  c = 0.005~\unit{\metre}, 
	 m = 0.0844~ \unit{\kilogram}, a = \SI{9.81}{\meter / \second\squared}.
$

We seek to maintain the ball at a target position $y_2^* = \SI{0.002}{\meter}$, that is, the equilibrium to stabilize is $y^* = \begin{bmatrix} \sqrt{2kma} & y_2^* & 0 \end{bmatrix}$. Define the error state $x := y - y^*$, then the error system can be written as 
\be \label{eq:sysmls}
\dot x = f(x) + g(x)u
\ee
with 
$
f(x) := \begin{bmatrix} 
	     -\frac{\gamma}{k}(c - x_2 - y_2^*)(x_1 + y_1^*) \\
	     \frac{1}{m} x_3  \\
	     \frac{1}{2k}x_1^2 + \frac{y_1^*}{k} x_1
		\end{bmatrix} $, and
		$
g := \begin{bmatrix} 
	   1 \\
	   0 \\
	   0 
	 \end{bmatrix}.
$
Note that a full rank left annihilator of $g$ can be selected as 
$
g^\perp = \begin{bmatrix} 
	       0 & 1 & 0 \\
	       0 & 0 & 1
		  \end{bmatrix}. 
$
We pick a constant matrix $F_d \in \R^{3 \times 3}$ with $F_d + F_d^T \preceq 0$, and consider the matching equation $g^\perp F_d \nabla H_d(x) =  g^\perp f(x)$. Specifically, the matrix $F_d$ is parameterized as
\begin{equation}
F_d = \begin{bmatrix}
		\alpha_{11} & \alpha_{12} & \alpha_{13} \\
		v_{11}	    & v_{12}      &  v_{13}  \\
		v_{21}      & v_{22}      &  v_{23}  
	  \end{bmatrix}.
\end{equation}
The subsequent derivation is based on the design procedure \updt{shown in Fig.~\ref{fig:flowchart}.}

\emph{Step 1.} Since $g^\perp F_d$ is constant in this case, Theorem~\ref{thm:cnst} can be exploited to check if the system \eqref{eq:sysmls} achieves maximum energy shapeability. \updt{Condition~\eqref{cond:mes2} in Theorem~\ref{thm:cnst} reduces to 
$
\frac{1}{m} v_{23} - \frac{y_1^* + x_1}{k} v_{11} = 0.
$
Therefore, the solvability of the matching equation is ensured by setting $v_{23} = v_{11} = 0$. For simplicity, we set $\alpha_{12} = 0$, $v_{21} = -\alpha_{13}$, and $v_{22} = -v_{13}$. As a result, $\frac12(F_d + F_d^T) = \diag(\alpha_{11}, v_{12}, 0)$ and we set $\alpha_{11} < 0$ and $v_{12} < 0$ to ensure closed-loop stability.} Then the matching equation reduces to 
\be \label{eq:matmls}
\begin{split}
v_{12} \frac{\partial H_d}{\partial x_2}(x) + v_{13} \frac{\partial H_d}{\partial x_3}(x)	&= \frac{1}{m}x_3 \\
-\alpha_{13} \frac{\partial H_d}{\partial x_1}(x) - v_{13} \frac{\partial H_d}{\partial x_2}(x)	& = \frac{1}{2k}x_1^2 + \frac{y_1^*}{k} x_1.
\end{split}
\ee
Furthermore, the values of the nonzero parameters in $F_d$ are picked such that $F_d$ is nonsingular, and this ensures that Condition~\eqref{cond:mes1} in Theorem~\ref{thm:cnst} holds. Therefore, the system has maximum energy shapeability in this case.   

\emph{Step 2.} Consider the PDE \eqref{eq:homopde1} in this case, i.e., 
$
g^\perp F_d \nabla \xi(x) = 0.
$
Note that Condition~\eqref{cond:mes3} in Theorem~\ref{thm:poin} trivially holds since $F_d^{-1}g$ is constant. Therefore, the characteristic coordinate, is $\xi(x) = g^T F_d^{-T} x$. 

\emph{Step 3.} Based on Step 2, we can introduce the change of coordinate $z = \begin{bmatrix} \xi & \eta^T \end{bmatrix}^T$ with 
$
	\eta = \begin{bmatrix} 
		0 & 1 & 0 \\
		0 & 0 & 1
	\end{bmatrix} F_d^{-T} x.
$
Hence,~$z = F_d^{-T} x$. By Lemma~\ref{lem:nabeta}, we have $\nabla_\eta \bar H_d(z) = \rho(z)$ with 
\begin{equation} \label{eq:nHeta}
	\rho(z) = \begin{bmatrix} \frac{1}{m}(\alpha_{13} z_1 + v_{13} z_2) \\
	\frac{1}{2k}(\alpha_{11} z_1 - \alpha_{13} z_3)^2   + \frac{y_1^*}{k} (\alpha_{11} z_1 - \alpha_{13} z_3) \end{bmatrix}.
\end{equation}
\updt{Note that $\frac{\partial^2 \bar H_d}{\partial \eta^2}(0) = \frac{\partial \rho}{\partial \eta } (0)\succ 0$ with $v_{13}>0$ and $\alpha_{13}  < 0$}.

\emph{Step 4.} For the system~\eqref{eq:sysmls}, i.e., a single input system, the mapping $M(z)$ defined in Theorem~\ref{thm:main} is a scalar. Therefore, Eq.~\eqref{eq:Msym} and the PDE~\eqref{eq:Mxi} holds naturally in this case. Eq.~\eqref{eq:Meta} in this case reduces to  
\begin{equation} \label{eq:Metamls}
	\frac{\partial M}{\partial z_3} = \frac{\alpha_{11}^2}{k}.
\end{equation}
A particular solution to~\eqref{eq:Metamls} is 
$
	M(z) = p_1 + p_2 z_1^2 + \frac{\alpha_{11}^2}{k} z_3
$,
where $p_1, p_2 >0$ are the control gains to be selected. Note that~\eqref{eq:Metamls} is very easy to solve compared to~\eqref{eq:matmls}. According to Theorem~\ref{thm:m12}, there exists a large enough $p_1$ such that $\frac{\partial^2 \bar H_d}{\partial z^2}(0) \succ 0$, and equivalently $\frac{\partial^2 H_d}{\partial x^2}(0) \succ 0$. 

\emph{Step 5.} By Corollary~\ref{coro:int}, we have 
\begin{equation} \label{eq:nHz1}
\begin{split}
	\nabla_{z_1} \bar H_d(z)  =& \left ( \int_0^1 
	\begin{bmatrix} 
	M(\lambda z) &  \frac{\partial^T \rho}{\partial z_1} (\lambda z)
	\end{bmatrix} \mathrm{d} \lambda \right ) z \\
	=& p_1 z_1  + \frac{p_2}{3} z_1^3 + \frac{\alpha_{11}^2}{k} z_1 z_3  + \frac{\alpha_{13}}{m}z_2 \\
	&+ \frac{ \alpha_{11} y_1^*}{k}z_3 -\frac{\alpha_{11}\alpha_{13}}{2k}z_3^2.
\end{split}
\end{equation}
Combining~\eqref{eq:nHz1} with~\eqref{eq:nHeta} yields the expression of $\nabla_z \bar H_d(z)$. By~\eqref{eq:uinz}, the resulting control law is obtained as
\begin{align}
	u =& \frac{\gamma}{k}(c - x_2 - y_2^*)(x_1 + y_1^*) + p_1 z_1  + \frac{p_2}{3} z_1^3 \label{eq:conlaw}  \\
	& + \frac{\alpha_{11}^2}{k} z_1 z_3  + \frac{\alpha_{13}}{m}z_2 
	+ \frac{ \alpha_{11} y_1^*}{k}z_3 -\frac{\alpha_{11}\alpha_{13}}{2k}z_3^2. \nonumber
\end{align} 

Based on the above analysis, the control parameters are selected as: 
$\alpha_{11} = -2, \alpha_{13}= -2, v_{12} = -2, v_{13} = 2, p_1 = 400$, and $p_2 = 20$.
\updt{With the above parameter setting, $H_d(x)$ is strictly convex on 
\begin{equation*}
	\Omega := \left \{ x  = F_d^T z  ~\middle | 
	\begin{array} {l}
		z_1>-8.758 \times 10^{-4},
		z_3-z_1>-\frac{y_1^*}{2} \\
		z_2 + z_3 > -1.500 \times 10^{-3}
	\end{array}
	\right \}. 
\end{equation*}
Note that $0 \in \Omega$ and $\nabla  H_d(0)=0$. 
It is readily verified that 
	$ 
	\mathcal{L} \subset \Omega
	$
	where $\mathcal{L}:= \{ x \st H_d(x) \leq  8.266 \times 10^{-6} \}$.
Therefore, $\mathcal{L}$ is forward invariant under the closed-loop dynamics.
	Let $\mathcal{M} :=
	\{x \in \mathcal{L}  \st
	\nabla^\top H_d(x)R_d \nabla H_d(x)=0
	\}$. A straightforward calculation yields 
$
		\mathcal{M} = \{ x = F_d^T z \st z_1 =0, z_3 = 0\}.
$
Consider the closed-loop system $\dot x=F_d\nabla H_d(x)$, or equivalently,
$
\dot z=F_d^{-T}\nabla\bar H_d(z).
$
Note that
\begin{equation}
	F_d^{-T} \nabla\bar H_d(z) \big|_{z_1 =0, z_3 = 0}
	= \begin{bmatrix}
		0 &
		0 & 
		-\frac{z_2}{m}
	\end{bmatrix}^T.
\end{equation}
If $z_2\neq 0$, then $\dot z_3 \neq 0$, so the trajectory cannot remain in $z_3=0$ in this case. Therefore, we can conclude that
the largest invariant set contained in $\mathcal{M}$ is $\{0\}$. By \emph{LaSalle's invariance principle}, the equilibrium $x^*= 0$ is locally asymptotically stable, and $\mathcal{L}$ is a corresponding domain of attraction. }

\updt{
\subsection{Translational oscillator with rotational actuator}
The translational oscillator with rotational actuator (TORA)~\cite{arpenti2022constructive} is a mass-spring-like system where the mass is replaced by a cart equipped with a pendulum. It can be modeled as the pH system~\eqref{eq:mechph} with $q, p \in \R^2$, and 
\begin{align*}
		L(q)=&
		\begin{bmatrix}
			m_1 (\mu^2+r^2)  & m_1 \mu \cos q_1\\
			m_1 \mu \cos q_1  & m_1 + m_2
		\end{bmatrix} \\
V(q)=&\frac{k}{2} q_2^2 + m_1 \mu a (1-\cos (q_1)) \\
g =& \begin{bmatrix} 1 & 0 
\end{bmatrix}^T
\end{align*}
where $m_1$ and $r$ are the mass and radius of the pendulum, respectively, $m_2$ is the mass of the cart, $\mu$ is the length of the link between the cart and the pendulum, and $k$ denote the elastic coefficient of the spring. 
The states $q_1(t)$ and $q_2(t)$ are the angle of the pendulum w.r.t. the vertical and the displacement of the cart, respectively, and $p_1(t)$ and $p_2(t)$ the relative generalized momenta. 
It is worth noting that the TORA system belongs to a broad class of two-degree-of-freedom (2-DoF) underactuated mechanical systems~\cite{liu2013survey}, indicating the applicability of the proposed method to this class of systems.

As in~\cite{arpenti2022constructive}, the model parameters are chosen as
$m_1  = \SI{1}~\unit{\kilogram}$, $m_2 = \SI{10}~ \unit{\kilogram}$, $\mu = \SI{1}~\unit{\meter}$,
$r  = \SI{0.1}~\unit{\meter}$, $k = 5$, and $a = \SI{9.81}{\meter / \second\squared}$.
The desired equilibrium is $q^* = \begin{bmatrix} \pi & 0 \end{bmatrix}^T$ and $p^* = 0$. The IDA-PBC seeks to achieve the desired closed-loop system
\begin{equation*}
	\begin{bmatrix}\dot q\\\dot p\end{bmatrix}
	=\begin{bmatrix}
		0   & L^{-1} (q) L_d (q)\\
		-L_d(q) L^{-1}(q) & J_2(q,p)-G k_d G^T
	\end{bmatrix}
	\nabla H_d(q,p)
	\label{eq:target}
\end{equation*}
where $k_d >0$, $J_2(q,p)=\begin{bmatrix} 0 & j_2(q,p) \\ -j_2(q,p) & 0 \end{bmatrix}$, and $H_d(q,p)=\frac12 p^T L_d^{-1}(q) p + V_d(q)$. Here, $L_d(q)$ and $V_d(q)$ denote the desired inertia matrix and the desired potential energy, respectively. The IDA-PBC controller is 
\be \label{eq:torau}
u = g^\dagger(\nabla_q H - L_d L^{-1} \nabla_q H_d + J_2 L_d^{-1} p) - k_d g^T \nabla_p H_d.
\ee

Pick $g^\perp = \begin{bmatrix}0 & 1 \end{bmatrix}^T$. As discussed in Remark~\ref{re:underone}, the matching equation can be separated into the kinetic energy PDE
\be \label{eq:kepde}
\begin{aligned}
& g^\perp ( \nabla_q (p^T L^{-1} (q) p) ) - L_d(q) L^{-1}(q) \nabla_q (p^T L_d^{-1}(q) p)  \\
&~ + 2 J_2(q,p) L_d^{-1} p) = 0
\end{aligned}
\ee 
and the potential energy PDE
\begin{equation} \label{eq:pepde}
	g^\perp (\nabla_q V(q) - L_d(q) L^{-1}(q) \nabla_q V_d(q) ) = 0.
\end{equation}
Note that~\eqref{eq:pepde} only depends on $q$.
Let~$l_{ij}, l_{d,ij}$ denote the~$ij$th entry of~$L$ and $L_d$, respectively. Pick $\sigma, \alpha, \delta>0$ and $L_d$ with
\begin{equation}
\begin{aligned} \label{eq:ld}
l_{d,12}(q_1)
&=  l_{d,21}(q_1) = \sigma l_{11}
+\frac{l_{12}(q_1)}{l_{22}}
\left[\alpha-\sigma l_{12}(q_1)\right]\\
l_{d,11}(q_1) &= \frac{l_{d,12}^2(q_1)}{ \alpha} + \delta, ~
l_{d,22} = \alpha 
\end{aligned}
\end{equation}
then~\eqref{eq:kepde} is transformed into an algebraic equation w.r.t. $J_2(q,p)$.
Note that $L_d(q) \succ 0$ with the above parameter setting.
A straightforward computation yields
\be \label{eq:j2}
j_2(q,p) = - \frac{\sigma  m_1 \mu }{\alpha l_{22}} p_2   \sin (q_1)   (  \alpha-2\sigma l_{12}(q_1)   ).    
\ee
Using~\eqref{eq:ld}, the potential energy PDE~\eqref{eq:pepde} can be written as
\be \label{eq:pepdenew}
\begin{bmatrix}
	\sigma & \frac{\alpha-\sigma m_1 \mu \cos(q_1)}{l_{22}} 
\end{bmatrix}
 \frac{\partial V_d(q)}{\partial q} = k q_2.
\ee

The subsequent derivation is based on the design procedure as shown in~Fig.~\ref{fig:flowchart}. The corresponding homogeneous equation of~\eqref{eq:pepdenew} admits a solution
\be 
\xi(q)=q_2-   \frac{\alpha}{\sigma l_{22}}   (q_1 -  \pi  )  +  \frac{m_1 \mu}{l_{22}}   \sin (q_1).
\ee 
Let $\eta(q) := q_1 - \pi$ and define the change of coordinate~$z = \tau(q)$, with $\tau(q) :=\begin{bmatrix} \xi(q) & \eta(q) \end{bmatrix}^T$. Define $\bar V_d(z) := V_d(\bar \tau(z))$. By Lemma~\ref{lem:nabeta}, 
\be 
\rho(z) := \nabla_\eta \bar V_d(z) 
=\frac{k}{\sigma}\left(\xi+ \frac{\alpha}{\sigma l_{22}} \eta +  \frac{m_1  \mu}{l_{22}} \sin (\eta) \right).  
\ee
By Theorem~\ref{thm:main}, the counterpart of~\eqref{eq:Msym}-\eqref{eq:Meta} in this case reduces to 
$
\frac{\partial M }{\partial \eta}(z) = 0.
$
Therefore, we can set $M$ as a constant. By Corollary~\ref{coro:int}, we have
\be \label{eq:vdbar}
\bar V_d(z)
=\frac{M}{2}\xi^2
+\frac{k}{\sigma}\xi\eta
+  \frac{\alpha k}{   2\sigma^2 l_{22}}  \eta^2
+  \frac{m_1  \mu k}{\sigma l_{22}} (1-\cos(\eta)).
\ee
Note that $M > \frac{k l_{22}}{ \alpha}$ ensures that
$\bar V_d(z)$ is strictly convex with $|\eta| \leq \frac{\pi}{2}$. Hence,~$0 = \argmin \bar V_d(z)$, or equivalently, $q^* = \argmin V_d(q)$ in this region.
Then the control law~\eqref{eq:torau} can be readily obtained using \eqref{eq:ld}, \eqref{eq:j2}, and \eqref{eq:vdbar}. 
}

\updt{
For comparison, simulations are performed using the proposed method and the method in~\cite{arpenti2022constructive} under identical model settings. The parameters of the proposed controller are selected as $\sigma = 0.4, \alpha = 1.4, \delta = 2.2, M = 41.6$, and $k_d = 6.6$.
It can be shown that a domain of attraction with the proposed controller is $\{x~|~ H_d(q,p) \leq 1.702\}$ using a similar argument as in Section~\ref{sec:simu_mls}. 
The control parameters of the controller in~\cite{arpenti2022constructive} are chosen as~$k_1  = 1, k_2 = 0.145, k_3 = -0.9, k_5 = 3.5, k_d = 14$, and $k_4 = -\frac{l_{22} k }{2 k_3}$.

Two simulation cases are considered: the nominal case and the perturbed case with modeling and measurement uncertainties. For the nominal case, 
the system trajectories $q_1(t)$ and $q_2(t)$ are compared in Fig.~\ref{fig:exact_q}. For the perturbed case, we then introduce parameter uncertainty by increasing $m_1$, $m_2$, $\mu$, $r$, and $k$ by $20 \%$ from their nominal values. In addition, measurement noise is added to $q_1$, $q_2$, $p_1$, and $p_2$, with the corresponding noise variances set to $0.0125, 0.0025, 0.0125$, and $0.0025$, respectively. The corresponding trajectories are presented in Fig.~\ref{fig:robust_q}. For the two cases, the control effort quantified by $\int_0^t |u(\tau)|^2 \mathrm{d} \tau$, is compared in Fig.~\ref{fig:u2_int}. It can be observed the proposed controller achieves faster convergence while requiring less control effort.
}

\begin{figure}
	\begin{center}
		\includegraphics[scale=0.40]{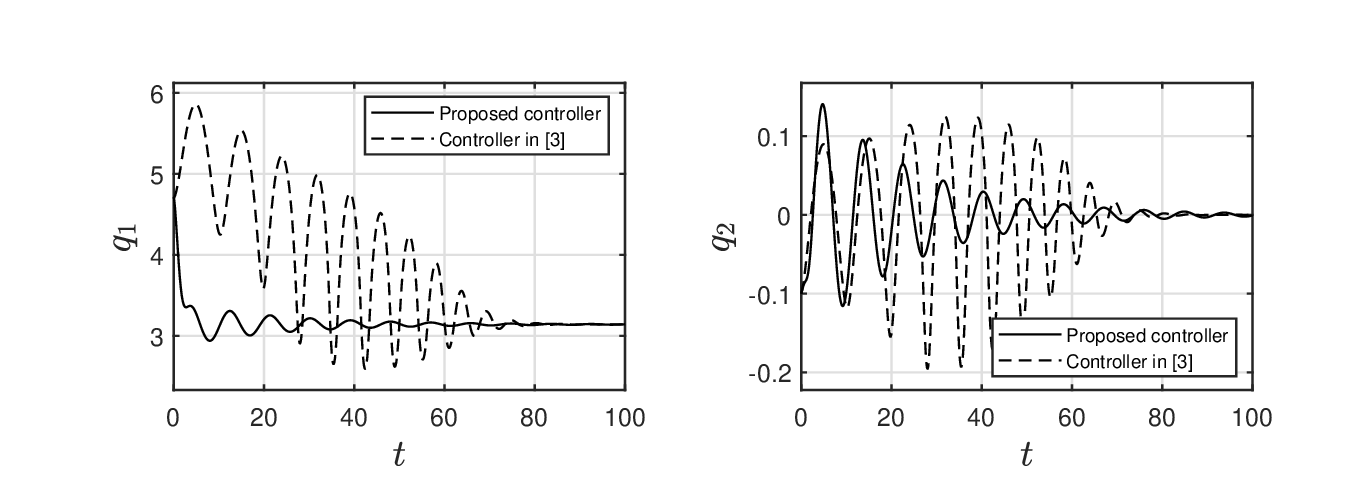}
		\caption{Comparison of the system trajectories $q_1(t)$ and $q_2(t)$ under the nominal model without measurement noise.}	\label{fig:exact_q}
	\end{center}
\end{figure}

\begin{figure}
	\begin{center}
		\includegraphics[scale=0.40]{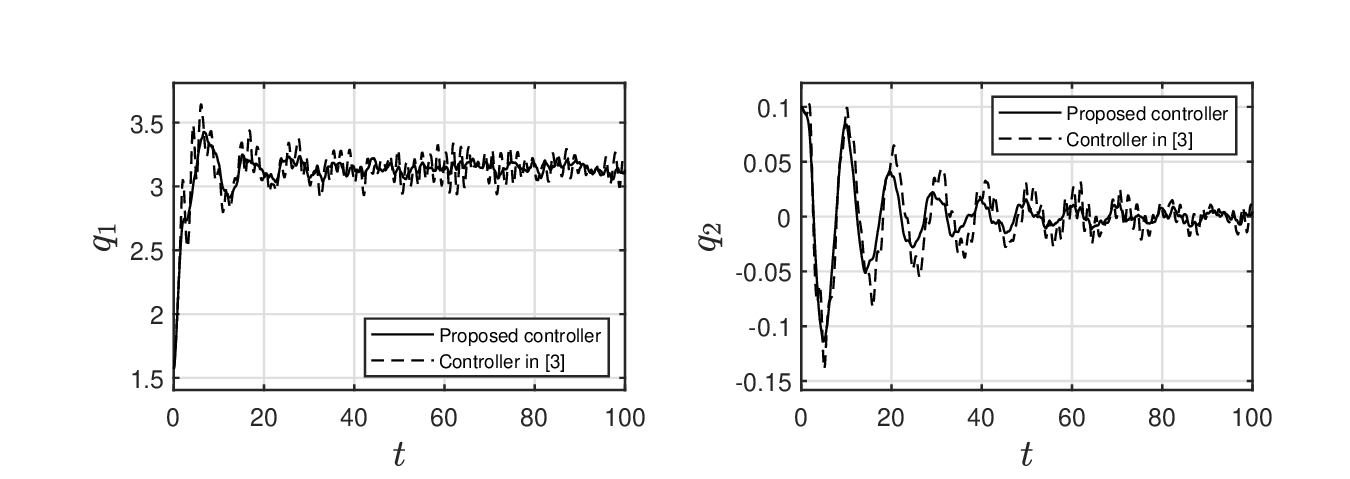}
		\caption{Comparison of the system trajectories $q_1(t)$ and $q_2(t)$ in the presence of parameter uncertainty and measurement noise.}	\label{fig:robust_q}
	\end{center}
\end{figure}

\begin{figure}[t]
	\begin{center}
		\includegraphics[scale=0.40]{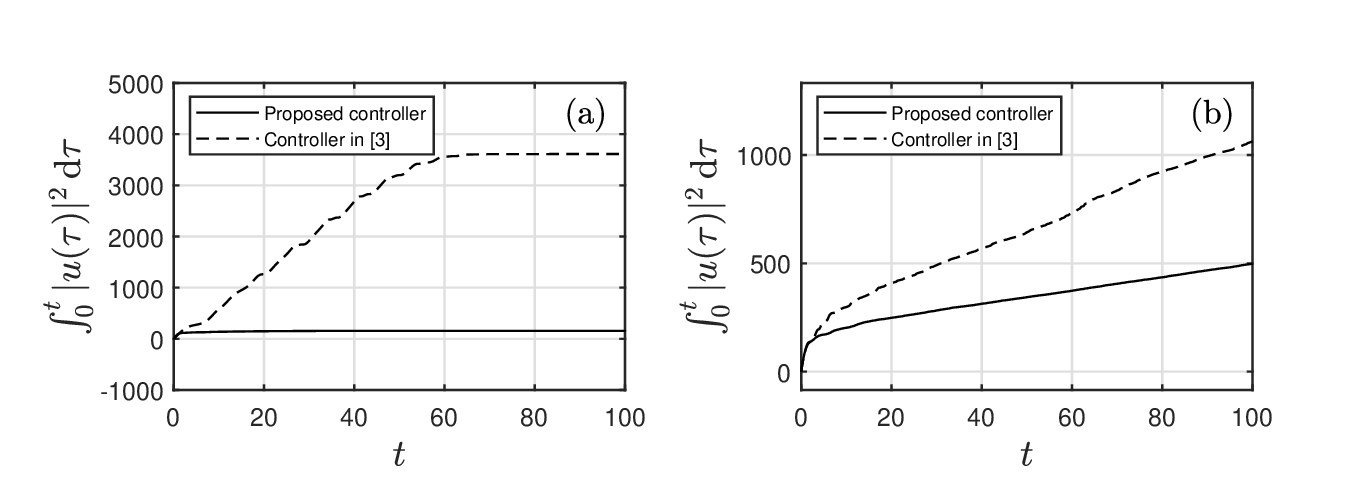}
		\caption{Comparison of the control effort: (a) the nominal case and (b) the perturbed case.}	\label{fig:u2_int}
	\end{center}
\end{figure}

\section{Conclusion}\label{sec:conclu}
This paper revisits the Interconnection and Damping Assignment Passivity-Based Control (IDA-PBC) method for affine nonlinear systems assuming the notion of \emph{maximum energy shapeability}. 
\updt{In this case, the search for a desired energy function $H_d$ can be recast as finding a $m \times m$ symmetric matrix-valued mapping $M$ subject to integrability conditions, with $m$ the dimension of the control input. The residual freedom in $M$ is used to satisfy the Hessian condition at the desired equilibrium. For the single-input case, the mapping $M$ reduces to a scalar-valued one and admits an explicit solution by successive integration.} 

\updt{Several issues remain open. First, for the general case with $m>1$, it remains unclear whether solving for $M$ is easier than solving the original matching equation.
Second, the sufficient conditions for maximum energy shapeability are stated for a given pair $(G^\perp,F_d)$. A systematic, e.g., algebraic or optimization-based, parametrization of $(G^\perp,F_d)$ that renders the distribution involutive is yet to be developed. Third, extending the procedure to the case where fewer than $m$ characteristic coordinates exist, such as certain mechanical systems with underactuation degree greater than one, would broaden its applicability. Finally, robustness against input disturbances and actuator dynamics, together with experimental validation, are natural next steps.}

\bibliographystyle{IEEEtranS}
\bibliography{maxida}
\end{document}